\documentclass[11pt]{amsart}

\usepackage{amssymb}
\usepackage{enumerate}
\usepackage{mathrsfs}
\usepackage{relsize}

\newtheorem{theorem}{Theorem}[section]

\newtheorem{lemma}[theorem]{Lemma}
\newtheorem{corollary}[theorem]{Corollary}

\theoremstyle{definition}

\newtheorem{question}[theorem]{Question}
\newtheorem{problem}[theorem]{Problem}

\theoremstyle{remark}

\numberwithin{equation}{section}

     \DeclareMathOperator{\Aut}{Aut}

    \DeclareMathOperator{\malg}{MALG}

    \DeclareMathOperator{\id}{Id}

    \DeclareMathOperator{\gp}{\bf GP}

    \DeclareMathOperator{\lo}{\bf LO}
    \DeclareMathOperator{\lm}{\mathsmaller{LM}}
    \DeclareMathOperator{\bm}{\mathsmaller{BM}}
    \DeclareMathOperator{\ma}{MA}
    \DeclareMathOperator{\mep}{mp}
    \DeclareMathOperator{\fin}{FIN}
    \DeclareMathOperator{\co}{co}

    \newcommand{\restrict}{\upharpoonright}

    \newcommand{\actson}{\curvearrowright}

\def\R{{\mathbb R}}

\def\N{{\mathbb N}}

\begin{document}

\title[On the pointwise implementation of near-actions]{On the pointwise implementation of near-actions}

\author{Asger T\"ornquist}
\address{Kurt G\"odel Research Center, University of Vienna, W\"ahringer Strasse 25, 1090 Vienna, Austria}
\email{asger@logic.univie.ac.at}
\thanks{Research supported by the Austrian Science Foundation FWF grant no. P19375-N18.}




\begin{abstract}
We show that the continuum hypothesis implies that every measure
preserving near-action of a group on a standard Borel probability space
$(X,\mu)$ has a pointwise implementation by Borel measure preserving
automorphisms.
\end{abstract}

\maketitle

\section{Introduction}

{\bf (A)} Let $(X,\mu)$ be a standard non-atomic Borel probability
space. A Borel bijection $f:X\to X$ is {\it measure preserving}
(m.p.) if for all Borel $A\subseteq X$, $\mu(f^{-1}(A))=\mu(A)$. In
this paper the group of measure preserving Borel bijections on
$(X,\mu)$ will be denoted $G_{\mep}(X,\mu)$. Let
$$
I_{\mep}(X,\mu)=\{f\in G_{\mep}(X,\mu): f(x)=x \text{ a.e.}\},
$$
that is, $I_{\mep}(X,\mu)$ is the (normal) subgroup of all those
measure preserving bijections that fix all but a null set of points
in $X$. The cosets in
$$
G_{\mep}(X,\mu)/I(X,\mu)
$$
are usually called {\it measure preserving transformations}, and
they form a group denoted $\Aut(X,\mu)$. The group
$\Aut(X,\mu)$ is a Polish group when equipped with the {\it weak}
topology, i.e. the topology it inherits when naturally identified
with a subgroup of the unitary group of $L^2(X,\mu)$, see
\cite{halmos56}. We will denote by $\bar\varphi$ the canonical
homomorphism of $G_{\mep}(X,\mu)$ onto $\Aut(X,\mu)$.

The present paper is concerned with the following fundamental
question:

\begin{question}[``The lifting problem for $\Aut(X,\mu)$'']
Is it possible to find a homomorphisms $h:\Aut(X,\mu)\to
G_{\mep}(X,\mu)$ such that $\bar\varphi\circ h(T)=T$ for all measure
preserving transformations $T$? That is, does the identity
automorphism $\id:\Aut(X,\mu)\to\Aut(X,\mu)$ ``split'',
$$
\id: \Aut(X,\mu)\overset{h}{\longrightarrow}
G_{\mep}(X,\mu)\overset{\bar\varphi}{\longrightarrow}\Aut(X,\mu).
$$\label{mainquestion}
\end{question}

We will call such a map $h$ a (homomorphic) {\it lifting} of the
identity automorphism on $\Aut(X,\mu)$. Our main result is:

\begin{theorem}
Assume that the Continuum Hypothesis, CH, holds. Then the identity
automorphism $\id:\Aut(X,\mu)\to\Aut(X,\mu)$ splits, that is, there
is a homomorphic lifting $h:\Aut(X,\mu)\to G_{\mep}(X,\mu)$ such
that $\bar\varphi\circ h=\id$.\label{mainthm1}
\end{theorem}

\bigskip

{\bf (B)} The motivation behind Question \ref{mainquestion} comes
from ergodic theory, where the group $\Aut(X,\mu)$ is of fundamental
interest. In the study of measure preserving dynamical systems, an
important distinction is made between {\it spatial actions} on the
one hand, and {\it near-actions} on the other hand. A spatial
measure preserving action of a group $G$ is simply an action
$\rho:G\times X\to X$ of $G$ on $X$ (sometimes written $\rho:
G\actson X$) in the usual sense, that additionally preserves the
measure, i.e. for all Borel $B\subseteq X$ we have
$$
\mu(B)=\mu(\rho(g)(B)).
$$
Here it is part of our assumption on $\rho$ that $\mu(\rho(g)(B))$
is $\mu$-measurable for each $g\in G$ and Borel $B\subseteq X$. If
we additionally assume that $\rho(g,\cdot):x\mapsto \rho(g,x)$ is a
Borel function for each $g\in G$ (that is, $\rho$ corresponds to a
homomorphism of $G$ into $G_{\mep}(X,\mu)$), then we will call
$\rho$ a \emph{spatial action by (measure preserving) Borel
automorphisms}.

In most cases of interest $G$ will be a Polish group, or at least a
topological group, in which case it may be natural to further
require that $\rho:G\times X\to X$ be Borel. We call such an action
$\rho$ a \emph{Borel spatial action} (following the conventions of
\cite{gltswe05}.)

On the other hand, a {\it near-action} of $G$ is a map $\eta:
G\times X\to X$ such that
\begin{enumerate}
\item If $e\in G$ is the identity element then $\eta(e)(x)=x$ for almost all
$x$.
\item For all $g,h\in G$, $\eta(gh)(x)=\eta(g)(\eta(h)(x))$ for almost
all $x$.
\item For all Borel $B\subseteq X$ and $g\in G$,
$\mu(B)=\mu(\eta (g)(B))$.
\end{enumerate}
Again, it is part of the definition of a near-action that
$\eta(g)(B))$ is measurable for all $g\in G$. If $G$ is a
topological group we call $\eta$ a measure preserving {\it Borel
near-action} if $\eta$ is Borel as a function $\eta:G\times X\to X$.

The above notions of spatial action and near-action and the related
concepts may be generalized to their {\it measure class preserving} counterparts in the
obvious way.

There is a direct correspondence between the near-actions of a
group $G$ and the homomorphisms of $G$ into $\Aut(X,\mu)$: Each
element of $g\in G$ defines through $x\mapsto \eta(g)(x)$ a
measurable m.p. bijection almost everywhere, and thus an element of
$T_g\in \Aut(X,\mu)$. Condition (2) of the definition of a
near-action ensures that $g\mapsto T_g$ is a homomorphism, which is
uniquely determined by $\eta$. If on the other hand $f:
G\to\Aut(X,\mu)$ is a homomorphism, we obtain a near-action of
$G$ by picking a representative in $G_{\mep}(X,\mu)$ for $f(g)$, for
each $g\in G$. There are of course many near-actions corresponding
to $f$.

If $\eta:G\times X\to X$ is a near-action, we call a (Borel)
spatial action $\rho:G\times X\to X$ a (Borel) {\it spatial model} of
$\eta$ if
$$
\eta(g,x)=\rho(g,x) \text{ $\mu$ a.e. } x.
$$
The relationship between homomorphisms and liftings on the one hand,
and near-actions and Borel spatial models on the other hand is
clear: If $\eta:G\times X\to X$ is a near-action, and
$f:G\to\Aut(X,\mu)$ is the corresponding homomorphism, then $\eta$ has a spatial model if and only if there is a homomorphism $h:
G\to G_{\mep}(X,\mu)$ such that
$$
\bar\varphi\circ h=f.
$$
Thus similarly to Question \ref{mainquestion} we may ask:

\begin{question}
If $\eta:G\times X\to X$ is a m.p. near-action of a group $G$, does
$\eta$ have a spatial model $\rho$ acting by Borel automorphisms?
\end{question}

If above we require that $\eta$ is a Borel near-action, and ask for
a Borel spatial model, this is a well-studied question which has a
negative solution in general: It was shown by Glasner, Tsirelson and
Weiss \cite{gltswe05} that there are Polish groups which do not
admit {\it any} non-trivial m.p. Borel spatial actions. A particular
instance of this is the group $\Aut(X,\mu)$ itself, but Glasner,
Tsirelson and Weiss' results encompass all so-called {\it Levy
groups}, see \cite{gltswe05} and \cite{glwe05}. On the other hand, a
classical theorem of Mackey \cite{mackey62} shows that if $G$ is a
locally compact group, then any Borel m.p. near-action of $G$ has a
Borel spatial model. This result was recently extended by
Kwiatkowska and Solecki \cite{kwso09} to apply to all isometry
groups of locally compact metric spaces.

However, in sharp contrast to the Glasner-Tsirelson-Weiss result, we
will show the following:

\begin{theorem}
Suppose $G$ is a group of cardinality at most $\aleph_1$. Then every m.p.
near-action of $G$ on a standard Borel probability space $(X,\mu)$
has a spatial model which acts by Borel m.p. automorphisms. Equivalently,
if $f:G\to\Aut(X,\mu)$ is a homomorphism, then there is a
homomorphism $h:G\to G_{\mep}(X,\mu)$ such that $\bar\varphi\circ
h=f$. In particular, if CH holds then all m.p. near-actions have a
spatial model acting by m.p. Borel automorphisms.\label{mainthm2}
\end{theorem}

Of course, Theorem \ref{mainthm1} is a direct consequence of Theorem
\ref{mainthm2}. Theorem \ref{mainthm2} holds more generally for measure class preserving near-actions, see (C) below.

\medskip

It is natural to also consider the easier question of simply having
a (measurable) spatial model of a near-action. Here it turns out
that CH can be replaced with Martin's Axiom:

\begin{theorem}
Assume that Martin's Axiom, MA, holds. Then every m.p. near-action
on a standard Borel probability space has a spatial model.
\label{leb}
\end{theorem}

\bigskip

{\bf (C)} After reading a note circulated by this author early in
the summer of 2009 and containing the proof of Theorem
\ref{mainthm1}, David Fremlin was kind enough to point out to the
author that his proof shows something much more general than Theorem
\ref{mainthm1}.

Indeed, let $X$ be a standard Borel space, and let $\mathcal B(X)$
denote the $\sigma$-algebra of Borel sets on $X$. Let $\mathcal J$
be an ideal in $\mathcal B(X)$. Then define $G(X,\mathcal J)$ to be
the group consisting of those Borel bijections $f:X\to X$ such that
$f^{-1}(B)\in \mathcal J$ if and only if $B\in\mathcal J$, for all
$B\in \mathcal J$. Then we may form the group
$$
I(X,\mathcal J)=\{f\in G(X,\mathcal J): \{x\in X:f(x)\neq
x\}\in\mathcal J\}.
$$
Clearly $I(X,\mathcal J)$ is a normal subgroup of $G(X,\mathcal J)$,
and so we may consider the quotient group
$$
\Aut(X,\mathcal J)=G(X,\mathcal J)/I(X,\mathcal J).
$$
Denote by $\varphi_{\mathcal J}:G(X,\mathcal J)\to \Aut(X,\mathcal
J)$ the corresponding homomorphism with kernel $I(X,\mathcal J)$.
Then Theorem \ref{mainthm2} has the following generalization:

\begin{theorem}
Suppose $\mathcal J$ is a $\sigma$-ideal in $\mathcal B(X)$ such
that there is at least \emph{one} uncountable Borel set $B\in
\mathcal J$. Then if $G\leq \Aut(X,\mathcal J)$ is any subgroup at
most of cardinality $\aleph_1$, then there is a homomorphism $h:
G\to G(X,\mathcal J)$ such that $\varphi_{\mathcal J}\circ
h=\id_G$.\label{mainthm3}
\end{theorem}

If we take $\mathcal J$ to be the $\sigma$-ideal of $\mu$-null sets
of some standard Borel probability measure on $X$, then
we obtain the obvious generalization of Theorem \ref{mainthm2} to
{\it measure class} preserving near actions and spatial actions.
Another case of interest where Theorem \ref{mainthm3} applies is if
we take $X$ to be a Polish space, and let $\mathcal J$ be the ideal
of meagre sets.

\bigskip

{\bf (D)} Question \ref{mainquestion} belongs to a line of set
theoretic research into quotient structures and their homomorphisms.
It is most closely related to the lifting problem for the measure
algebra: Recall that if $\mathcal B(X)$ is the set of Borel subsets of
some standard Borel probability space $(X,\mu)$, and $\mathcal I(X)$
is the ideal of $\mu$ null sets in $\mathcal B(X)$, then
$$
\malg(X,\mu)=\mathcal B(X)/\mathcal I(X).
$$
Let $\tilde\varphi: \mathcal B(X)\to\malg(X,\mu)$ be the canonical
boolean algebra homomorphism with kernel $\mathcal I(X)$. It was
shown by von Neumann and Stone in \cite{vnst35} that under CH there
is a boolean algebra homomorphism $h:\malg(X,\mu)\to\mathcal B(X)$
(a ``lifting'') such that $\tilde\varphi\circ h$ is the identity on
$\malg(X,\mu)$, i.e that the identity homomorphism
$\id:\malg(X,\mu)\to\malg(X,\mu)$ `splits' over $\tilde\varphi$.
Later Shelah famously showed in \cite{shelah83} that it is
consistent with ZFC and $2^{\aleph_0}=\aleph_2$ that
$\id:\malg(X,\mu)\to\malg(X,\mu)$ {\it does not} split. On the other
hand, Carlson, Frankiewicz and Zbierski showed in \cite{cafrzb94}
that it is consistent with $2^{\aleph_0}=\aleph_2$ that there {\it
is} a lifting.

There has been many other interesting results regarding the
structure of quotient objects, perhaps the most well-known being the
structure of the automorphisms of the boolean algebra $\mathcal
P(\omega)/\fin$, see \cite{shelahpif}, \cite{shst88},
\cite{velickovic93}, \cite{farah00} and \cite{farah05}. Some
remarkable recent results regarding the similar problem for the
automorphisms of the Calkin Algebra $\mathcal C(H)=\mathcal
B(H)/\mathcal K(H)$, where $H$ is an infinite dimensional separable
Hilbert space and $\mathcal K(H)$ is the ideal of compact operators
in the $C^*$ algebra of all bounded operators $\mathcal B(H)$, has
been achieved by Phillips and Weaver \cite{phwe07}, and by Ilijas
Farah in \cite{farah09}.

The results quoted above regarding the lifting problem for
$\malg(X,\mu)$ suggest several similar problems for the lifting
problem for $\Aut(X,\mu)$. In the last section of this paper we
discuss these and other open problems that promise to be quite
interesting for future research.

\bigskip

{\bf (E)} The paper is organized into four sections, including the
introduction. \S 2 is dedicated to the proof of Theorem
\ref{mainthm3}. In \S 3 we prove Theorem \ref{leb}, that assuming
Martin's Axiom there is a Lebesgue measurable lifting of
$\Aut(X,\mu)$. Finally, in \S 4 we discuss several open questions
related to the theme of this paper.

\bigskip

{\it Acknowledgements.} I would like to thank David Fremlin for
carefully reading a very preliminary version of this paper, and for
pointing out the right generality in which the main theorem should
be stated. I would also like to thank Katherine Thompson for many
useful discussions about the ideas for this paper.
\section{Proof of the main theorem}

In this section we will prove Theorem \ref{mainthm3}, from which
Theorems \ref{mainthm1} and \ref{mainthm2} follow. We fix once and
for all a standard Borel space $X$ and a $\sigma$-ideal $\mathcal J$
in $\mathcal B(X)$. Denote by $\co(\mathcal J)$ the filter of all
$Z\in\mathcal B(X)$ such that $X\setminus Z\in\mathcal J$. Recall
that $\varphi_{\mathcal J}: G(X,\mathcal J)\to\Aut(X,\mathcal J)$
denotes the canonical epimorphism with kernel $I(X,\mathcal J)$.

The proof of Theorem \ref{mainthm3} requires several lemmata. As a
point of departure let us note the following easy fact:

\begin{lemma}
Let $H<\Aut(X,\mathcal J)$ be a countable group, and let
$T_0\in\Aut(X,\mathcal J)\setminus H$. Suppose $\rho:H \times X\to
X$ is an action by Borel automorphisms such that
$\varphi_{\mathcal J}(\rho(T,\cdot))=T$ for all $T\in H$. If
$\varphi_{\mathcal J}(\theta_0)= T_0$ then there is Borel set
$Z\subseteq X$ such that $Z\in\co(\mathcal J)$, $Z$ is invariant
under $\rho$ and $\theta_0$, and there is an action $\tilde
\rho:\langle H\cup\{T_0\}\rangle\times Z\to Z$ such that
$\tilde\rho(T_0,\cdot)=\theta_0\restrict Z$ and $\tilde
\rho\restrict H\times Z = \rho\restrict H\times Z$.\label{extend}
\end{lemma}
\begin{proof}
Let $Z$ be the set of all $x\in X$ such that for all $n_1,\ldots,
n_k\in \{-1,1\}$, $n_0,n_{k+1}\in \{-1,0,1\}$ and
$\gamma_1,\ldots,\gamma_k\in H$ it holds that if
$$
T_0^{n_0}\gamma_1 T_0^{n_1}\cdots\gamma_k T_0^{n_{k+1}}=\id
$$
then
$$
\theta_0^{n_0}\rho(\gamma_1,\cdot)\theta_0^{n_1}\cdots\rho(\gamma_k,\cdot)
\theta_0^{n_{k+1}}(x)=x.
$$
Clearly $Z$ is Borel, $Z\in\co(\mathcal J)$ and $Z$ is invariant
under $\rho$ and $\theta_0$. For $z\in Z$ and $T\in\langle
H\cup\{T_0\}\rangle$ write $T=T_0^{n_0}\gamma_1
T_0^{n_1}\cdots\gamma_k T_0^{n_{k+1}}$, with
$\gamma_1,\ldots,\gamma_k\in H$, and define
$$
\tilde\rho(T,z)=\theta_0^{n_0}\rho(\gamma_1,\cdot)\theta_0^{n_1}\cdots\rho(\gamma_k,\cdot)
\theta_0^{n_{k+1}}(z).
$$
If $T=T_0^{n_0'}\gamma_1' T_0^{n_1'}\cdots\gamma_l' T_0^{n_{l+1}'}$
is another representation of $T$ of this form then by the definition
of $Z$ we must have
$$
\tilde\rho(T,z)=\theta_0^{n_0'}\rho(\gamma_1',\cdot)\theta_0^{n_1'}\cdots\rho(\gamma_l',\cdot)
\theta_0^{n_{l+1}'}(z).
$$
Thus $\tilde\rho$ is well-defined on $Z$. Clearly $\tilde\rho$ is a
$\langle H\cup\{T_0\}\rangle$-action by Borel bijections such that
$\tilde\rho(T_0,\cdot)=\theta_0\restrict Z$ and $\tilde
\rho\restrict H\times Z = \rho\restrict H\times Z$.
\end{proof}

If the conclusion of Lemma \ref{extend} was always true for $Z=X$
and for \emph{some} $\theta_0$ with $\varphi_{\mathcal
J}(\theta_0)=T_0$, then Theorem \ref{mainthm3} would follow from an
easy transfinite induction. However, we certainly cannot expect to
have Lemma \ref{extend} with $Z=X$ for just any $\theta_0$ with
$\varphi_{\mathcal J}(\theta_0)=T_0$ if we proceed na\"ively,
loosely speaking since the action of $H$ might behave very
differently on some $\mathcal J$ set than on a given $\co(\mathcal
J)$ set.

We will prove a Lemma that shows that under certain conditions,
there is always {\it some} $\theta_0$ with $\varphi_{\mathcal
J}(\theta_0)=T_0$ that satisfies Lemma \ref{extend} with $Z=X$.
First we need to recall the notion of a (Bernoulli) shift action,
and a Lemma regarding their universality among Borel actions of
countable groups. If $H\leq G$ are countable groups and $X$ is a
standard Borel space, recall that a Borel action $\beta:H\actson
X^G$ is defined by
$$
(\beta(h)(x))(g)=x(h^{-1}g)
$$
for $h\in H$ and $g\in G$.

\begin{lemma}[Folklore]
Let $H\leq G$ be countable groups and let $X=2^\N=\{0,1\}^\N$ be
Cantor space. Then the shift action $\beta: H\actson X^G$ has the
following universality property: If $Y$ is a standard Borel space
and $\sigma:H\actson Y$ is a Borel action of $H$ on $Y$, then there
is a Borel injection $\psi:Y\to X^G$ such that for all $h\in H$,
$$
\psi(\sigma(h)(x))=\beta(h)(\psi(x)),
$$
i.e., $\psi$ conjugates the actions $\sigma$ and
$\beta|\psi(Y)$.\label{universal}
\end{lemma}
\begin{proof}
It is easy to see that $\beta:H\actson X^G=(2^{\N})^G$ is conjugate
(via a Borel map) to $\beta':H\actson X^H$, so we may assume that
$H=G$. Let $(B_n)_{n\in\N}$ be a sequence of Borel sets that
separate the points if $Y$. Define $\psi:Y\to X^H$ by
$$
\psi(x)(h)(n)=\left\{ \begin{array}{ll}
1 & \textrm{if $\sigma(h^{-1})(x)\in B_n$}\\
0 & \textrm{otherwise.}\\
\end{array} \right.
$$
Then
\begin{align*}
&\psi(\sigma(g)(x))(h)(n)=1 \iff \sigma(h^{-1}g)(x)\in B_n \iff \\
&\psi(x)(g^{-1}h)(n)=1 \iff \beta(g)(\psi(x))(h)(n)=1,
\end{align*}
which shows that $\psi$ conjugates the $\sigma$ and $\beta$ actions.
Finally, since $(B_n)$ separates points $\psi$ is 1-1.
\end{proof}

\medskip

We now have the following strong version of Lemma \ref{extend},
which will play a key role in our construction:

\begin{lemma}[``Rearrangement'']
Let $H\leq G$ be countable groups, and suppose there are countable
groups $G_i$, $i\in\N$, such that $G\leq G_i$ for all $i$. Let $X$
be a standard Borel space which is partitioned into Borel pieces,
$$
X=X_0\sqcup\bigsqcup_{i\in\N} (2^\N)^{G_i},
$$
that is, $X$ is the disjoint union of $X_0$ and $(2^\N)^{G_i}$,
$i\in\N$, $X_0$ is Borel, and  $(2^\N)^{G_i}$ carries its usual
Borel structure for all $i\in\N$. Suppose $\rho:H\actson X$ is a
Borel action of $H$ such that
$$
\rho\restrict H\times (2^\N)^{G_i}
$$
is the shift action. Then there is a Borel action $\hat\rho:G\actson
X$ such that $\hat\rho\restrict H\times X=\rho$.\label{onestep}
\end{lemma}
\begin{proof}
For convenience, let $X_i=(2^\N)^{G_i}$, so that $X=\bigsqcup_{i\geq
0} X_i$. Note that by definition, $X_0$ is $\rho$-invariant. Thus by
Lemma \ref{universal} we may find Borel injections $\psi_i:X_0\to
(2^\N)^{G_i}$, $i\geq 1$, such that $\psi_i$ conjugates
$\rho\restrict X_0$ and $\rho \restrict \psi_i(X_0)$. We let
$\psi_0:X_0\to X_0$ be $\psi_0(x)=x$. Now let
$$
Z_0=X_0\sqcup (X_1\setminus \psi_1(X_0)),
$$
and for $i\geq 1$ let
$$
Z_i=(X_{i+1}\setminus (\psi_{i+1}(X_0)))\sqcup \psi_{i}(X_0).
$$
Then each $Z_i$ is $\rho$-invariant, and the functions
$$
\bar\psi_i:Z_i\to X_{i+1}: x\mapsto\left\{ \begin{array}{ll}
\psi_{i+1}\psi_{i}^{-1}(x) & \textrm{if $x\in \psi_i(X_0)$}\\
x & \textrm{otherwise.}\\
\end{array} \right.
$$
defines a bijection between $Z_i$ and $X_{i+1}$, $i\geq 0$.
Moreover, $\bar\psi_i$ conjugates the actions $\rho\restrict Z_i$
and $\rho\restrict X_{i+1}$. Now for $i\geq 1$ let $\beta_i:G\times
X_i\to X_i$ be the Bernoulli shift action of $G$. For each $i\geq 0$
we may define a $G$-action $\hat\rho_i:G\times Z_i\to Z_i$ by
$$
\hat\rho_i(\gamma,x)=\bar\psi_i^{-1}(\beta_i(\gamma,\bar\psi_i(x))).
$$
Since $\beta_i\restrict H\times X_i=\rho\restrict H\times X_i$ and
$\bar\psi_i$ conjugates the actions $\rho\restrict Z_i$ and
$\rho\restrict X_{i+1}$, we have that $\hat\rho_i\restrict H\times
Z_i=\rho\restrict H\times Z_i$. Thus if we let
$$
\hat\rho=\bigcup_{i\in\N} \hat\rho_i,
$$
then $\hat\rho$ is a Borel $G$-action as required.
\end{proof}

\subsection{The master actions}

Fix a group $H\leq\Aut(X,\mathcal J)$ of cardinality $\aleph_1$. Let
$(T_\alpha:\alpha<\omega_1)$ be an enumeration of $H$. For
$\alpha<\omega_1$ define
$$
H_\alpha=\langle T_\beta:\beta<\alpha\rangle,
$$
i.e., $H_\alpha$ is the subgroup generated by all $T_\beta$,
$\beta<\alpha$. By definition we let $H_0=\{I\}$, the subgroup
containing only the identity in $\Aut(X,\mathcal J)$.

We now describe a natural Polish space $S_\alpha$ of countable
groups containing a copy of $H_{\alpha}$, and with a sequence of
designated elements that will correspond to the $T_\beta$,
$\beta<\alpha$.

If $x\in 2^{\N\times\N}$, write $R_x$ for the relation
$$
m R_x n\iff x(m,n)=1.
$$
We let
$$
\lo=\{x\in 2^{\N\times\N}: R_x \text{ is a (strict) linear ordering}\}.
$$
Elements of $\lo$ will usually be denoted by $<^*$ or something
similar, and we will write $m<^*n$ to indicate what should more correctly be
written $m R_{<^*} n$. Moreover, we will write $m\leq^* n$ if
$m<^*n$ or $m=n$. Let $\gp$ denote the set of all groups with
underlying set $\N$, i.e.
\begin{align*}
\gp=\{&(f,g,e)\in \N^{\N\times\N}\times\N^\N\times\N:\\
&(\forall i,j,k\in N)
(f(f(i,j),k)=f(i,f(j,k)))\wedge\\
&(\forall i\in \N) f(i,e)=f(e,i)=i \wedge (\forall i)
f(i,g(i))=e\}.
\end{align*}
Again, it is notationally convenient to denote an element in $\gp$
by $G$, and refer to the multiplication in $G$ by $\cdot_G$, the
inverse by $^{-1}$, and the identity element by $e_G$.

For $\alpha<\omega_1$ let
$$
A_{\alpha}=\{(<^*,n)\in\lo\times\N: (\{k\in\N: k<^*n\},<^*)\simeq\alpha\},
$$
that is, $(<^*,n)\in A_{\alpha}$ if and only if the initial segment
$\{k:k<^*n\}$ is order isomorphic to the ordinal $\alpha$.

\begin{lemma}
The set $A_\alpha$ is Borel for all
$\alpha<\omega_1$.\label{Borel-lo}
\end{lemma}

\begin{proof}
By induction on $\alpha$. If $\alpha=\beta+1$ then
\begin{align*}
A_{\alpha}=\{(<^*,n)\in\lo\times\N: &(\exists m) ((<^*,m)\in A_\beta\wedge\\
 &(\forall k) k<^*n\implies k\leq^* m)\}.
\end{align*}
If $\alpha$ is a limit, fix $\beta_i<\alpha$, $i\in\N$ such that
$\sup_{i\in\N}\beta_i=\alpha$. Then
\begin{align*}
A_{\alpha}=\{&(<^*,n)\in\lo\times\N: (\forall i)(\exists m) (<^*,m)\in A_{\beta_i}\wedge\\
 &(\forall k) k<^*n\implies (\exists i)(\exists l) (<^*,l)\in A_{\beta_i}\wedge k<^* l\}.
\end{align*}
\end{proof}

{\it Definition.} For $\alpha<\omega_1$ we define $\mathcal
S_{\alpha}\subseteq \lo\times\gp$ to consists of all $(<^*,G)\in
\lo\times\gp$ such that
\begin{enumerate}[($i$)]
 \item For some $n\in\N$, $(<^*,n)\in A_\alpha$;
 \item There is a monomorphism $\varphi: H_{\alpha}\to G$ such that $(<^*,\varphi(T_\beta))\in A_\beta$ for all $\beta<
\alpha$.
\end{enumerate}

\begin{lemma}
 The set $\mathcal  S_\alpha$ is Borel for all $\alpha<\omega_1$.
\end{lemma}
\begin{proof}
 We claim that $(<^*,G)$ satisfies (i) and (ii) is equivalent to the statement: There is $n\in \N$ such that
\begin{enumerate}[($i'$)]
 \item $(<^*,n)\in A_\alpha$,
 \item For all $\beta_0,\ldots\beta_l<\alpha$, $n_0,\ldots, n_l\in \{-1,1\}$ and $m_0,\ldots m_l\in\N$ it holds that
\begin{align*}
\text{(a)    } T_{\beta_0}^{n_0}T_{\beta_1}^{n_1}\cdots T_{\beta_l}^{n_l}=I\wedge (\forall i\leq l) (<^*,m_i)\in
A_{\beta_i}\implies\\
m_0^{n_0}\cdot_G m_1^{n_1}\cdots m_l^{n_l}=e_G
\end{align*}
and
\begin{align*}
\text{(b)    } T_{\beta_0}^{n_0}T_{\beta_1}^{n_1}\cdots T_{\beta_l}^{n_l}\neq I\wedge (\forall i\leq l) (<^*,m_i)\in
A_{\beta_i}\implies\\
m_0^{n_0}\cdot_G m_1^{n_1}\cdots m_l^{n_l}\neq e_G.
\end{align*}
\end{enumerate}
It is clear that ($i$) and ($ii$) imply ($i'$) and ($ii'$), and that ($i'$) and ($ii'$) are Borel conditions. To show
($i'$) and ($ii'$) imply ($i$) and ($ii$), it suffices to show that if $(<^*,G)$ satisfies ($i'$) and ($ii'$), then the
map $\varphi: H_\alpha\to G$ given by
\begin{align*}
&\varphi(x)=m_0^{n_0}\cdot_G m_1^{n_1}\cdots m_l^{n_l}\iff\\ &x=T_{\beta_0}^{n_0}T_{\beta_1}^{n_1}\cdots
T_{\beta_l}^{n_l}\wedge (\forall i\leq l) (<^*,m_0)\in A_{\beta_i}.
\end{align*}
defines a monomorphism with $\varphi(T_\beta)=m$ if and only if
$(<*,m)\in A_\beta$. That $\varphi$ is a well-defined function
follows easily from (a). If $\varphi(x)=e_G$ and
$$
x=T_{\beta_0}^{n_0}T_{\beta_1}^{n_1}\cdots T_{\beta_l}^{n_l},
$$
then $x=I$ follows from (b). Thus is $\varphi$ is a monomorphism as required.
\end{proof}

\medskip

{\it Remark.} For $(<^*,G)\in \mathcal S_\alpha$, the unique
monomorphism $\varphi: H_\alpha\to G$ satisfying ({\it ii}) in the
definition of $\mathcal S_\alpha$ will be called the {\it canonical}
monomorphism $H_\alpha\to G$.

\bigskip

{\it Definition.} For $\alpha<\omega_1$, the $\alpha$'th {\it master
action} $\sigma_\alpha: H_\alpha\actson \mathcal M_\alpha$ is
defined by
\begin{enumerate}
 \item $\mathcal M_\alpha=\mathcal S_\alpha\times (2^{\N})^{\N}$;
 \item For $\beta<\omega_1$,
\begin{align*}
& \sigma_\alpha(T_\beta)(<^*_0,G_0,x)=(<^*_1,G_1,y)\iff\\
& <^*_0=<^*_1\wedge G_0=G_1\wedge\\
&  (\forall m) ((<^*_0,m)\in A_\beta\implies (\forall n) y(n)=x(m^{-1}\cdot_{G_0} n)).
\end{align*}
\end{enumerate}
Note that $\sigma_\alpha\restrict \{(<^*,G)\}\times (2^\N)^\N$ is
isomorphic to the shift-action of $H_\alpha$ on $(2^\N)^\N$, when
$H_\alpha$ is identified canonically with its image under the
canonical monomorphism $\varphi:H_\alpha\to G$, and $G$ is
identified with its underlying set $\N$. So we think of $(2^\N)^\N$
as $(2^\N)^G$ and think of $\sigma_\alpha|\{(<^*,G)\}\times
(2^\N)^\N$ as the shift action $H_\alpha\actson\ (2^\N)^G$ for each
$(<^*,G)\in\mathcal S_\alpha$. It is clear that condition (2) of the
definition of $\sigma_\alpha$ defines the action uniquely since the
$T_\beta$, $\beta<\alpha$ generate $H_\alpha$. Moreover,
$\sigma_{\alpha}(T_\beta)$ is a Borel automorphism of $\mathcal
M_\alpha$ for each $\beta<\alpha$, and so $\sigma_\alpha$ defines a
Borel action of $H_\alpha$ on $\mathcal M_\alpha$. Finally note that
if $\beta<\alpha$ then $\mathcal M_\alpha\subseteq\mathcal M_\beta$
and that for the actions $\sigma_\alpha,\sigma_\beta$ we have
$$
\sigma_\alpha\restrict H_\beta\times \mathcal M_\alpha=\sigma_\beta\restrict H_\beta\times\mathcal M_\alpha,
$$
where again $H_\beta$ is canonically identified with a subgroup of
$G$ for each $(<^*,G)\in\mathcal S_\alpha\subseteq\mathcal S_\beta$.
The following is clear from the definitions:

\begin{lemma}
If $\alpha<\omega_1$ is a limit ordinal, it holds for the master
action $\sigma_\alpha$ that
$$
\sigma_\alpha=\bigcup_{\beta<\alpha} \sigma_{\beta}\restrict
\mathcal M_\alpha.
$$
\label{masterlimit}
\end{lemma}

We are now ready to prove Theorem \ref{mainthm3}. The idea is to
construct the homomorphism $h:H\to G(X,\mathcal J)$ by transfinite
induction on $\omega_1$ by constructing in stages $h_\alpha:
H_\alpha\to G(X,\mathcal J)$, at each stage making sure that on some
$\mathcal J$ set we have a large invariant ``reservoir'' of master
actions. This will allow us to use Lemma \ref{onestep} to extend the
homomorphism at each stage of the induction.

\begin{proof}[Proof of Theorem \ref{mainthm3}.]
Since $\mathcal J$ contains an uncountable Borel set we can assume
that $X$ has the form $X=X_0\sqcup 2^\N\times \mathcal M_0$
(disjoint union) such that $X_0\in\co(\mathcal J)$ . We construct by
induction on $\alpha<\omega_1$ homomorphisms $h_{\alpha}:H_\alpha\to
G(X,\mathcal J)$ and uncountable Borel sets $Y_\alpha\subseteq 2^\N$
such that
\begin{enumerate}
 \item $h_0(I)=\id$, $Y_0=2^\N$;
 \item $h_\alpha: H_\alpha\to G(X,\mathcal J)$ is a homomorphism such that $\varphi_{\mathcal J}(h_\alpha(T))=T$ for all
$T\in H_\alpha$;
 \item If $\beta<\alpha$ then $Y_\beta\supseteq Y_\alpha$ and $Y_\beta\setminus Y_\alpha$ is countable;
 \item For $(y,x)\in Y_\alpha\times\mathcal M_\alpha$ we have $h_\alpha(T)(y,x)=(y,\sigma_\alpha(T)(x))$ for all $T\in H_
\alpha$;
 \item If $\beta<\alpha$ then $h_\beta=h_\alpha\restrict H_\beta$.
 \end{enumerate}
If this can be done then we let
$$
h=\bigcup_{\alpha<\omega_1} h_\alpha.
$$
By condition (5), $h$ is a function, $h:H\to G(X,\mathcal J)$. By
condition (2) it is a homomorphism with $\varphi_{\mathcal J}(h(T))=
T$ for all $T\in\Aut(X,\mathcal J)$. Thus to prove Theorem
\ref{mainthm3} it suffices to construct $h_\alpha$ and $Y_\alpha$
satisfying (1)--(5).

\medskip

Case 1: $\alpha$ is a limit ordinal.

\medskip

We let $h_\alpha=\bigcup_{\beta<\alpha} h_\beta$ and
$Y_\alpha=\bigcap_{\beta<\alpha} Y_\beta$. (1)--(3) and (5) are
clearly satisfied. (4) follows by Lemma \ref{masterlimit}.

\medskip

Case 2: $\alpha=\beta+1$.

\medskip
Let $\rho_\beta:H_\beta\times X\to X$ be the action corresponding to
$h_\beta$,
$$
\rho_\beta(T,x)=h_\beta(T)(x).
$$
By Lemma \ref{extend} we may find a Borel $Z\subseteq X_0$ and
$\theta:Z\to Z$ Borel such that $Z\in\co(\mathcal J)$ and
$\varphi_{\mathcal J}(\theta)= T_\beta$ (more precisely, an
extension $\bar\theta$ of $\theta$ satisfies $\varphi_{\mathcal
J}(\bar\theta)=T_\beta$), such that $Z$ is invariant under
$$
\{h_\beta(T):T\in H_\beta\}\cup \{ \theta \},
$$
and $\rho_\beta\restrict Z$ extends to an action
$\rho^0:H_\alpha\times Z\to Z$ such that
$$
\rho^0\restrict H_\beta\times Z=\rho_\beta\restrict H_\beta\times Z,
$$
and $\rho^0(T_\beta,\cdot)=\theta$.

Pick a countable sequence $(y_i\in Y_\beta:i\in\N)$ of distinct
elements in $Y_\beta$. Let
$Y_\alpha=Y_\beta\setminus\{y_i:y\in\N\}$, and let
$$
W=X\setminus(Z\cup Y_\alpha\times\mathcal M_\alpha).
$$
Also pick a sequence $(<^*_i,G_i)\in \mathcal S_\alpha, i\in\N,$ distinct. Then
$$
W=(W\setminus\bigcup_{i\in\N} \{(y_i,<^*_i, G_i)\}\times
(2^\N)^\N)\sqcup\bigsqcup_{i\in\N}  \{(y_i,<^*_i, G_i)\}\times
(2^\N)^\N.
$$
is a decomposition of $W$ into $\rho_\beta$-invariant pieces. Note
that $\rho_\beta\restrict \{(y_i,<^*_i, G_i)\}\times (2^\N)^\N$ is
isomorphic to a shift-action of $H_\beta< G_i$ on $(2^\N)^{G_i}$ and
by construction $H_\beta<H_\alpha\leq G_i$. Thus the hypothesis of
Lemma \ref{onestep} is satisfied for $\rho_\beta\restrict
H_\beta\times W$, and so we may find a Borel action $\rho^1:
H_\alpha\times W\to W$ such that $\rho^1\restrict H_\beta\times
W=\rho_\beta\restrict H_\beta\times W$.

Finally, we define $\rho^2:H_\alpha\times (Y_\alpha\times \mathcal
M_\alpha)\to Y_\alpha\times \mathcal M_\alpha$ by
$$
\rho^2(T,(y,x))=(y,\sigma_\alpha(T)(x)).
$$
Note that by (4), $\rho^2$ is an extension of $\rho_\beta\restrict
H_\beta\times (Y_\alpha\times \mathcal M_\alpha)$. It follows that
$$
\rho_\alpha=\rho^0\cup\rho^1\cup\rho^2
$$
defines a Borel action of $H_\alpha$ extending $\rho_\beta$. Let
$$
h_\alpha: H_\alpha\to G(X,\mathcal J): T\mapsto
\rho_\alpha(T,\cdot).
$$
Conditions (1), (3) and (5) are clearly satisfied. (2) follows since
$\rho^0(T_\beta,\cdot)=\theta$. (4) follows directly from the
definition of $\rho^2$. Thus the proof of Theorem \ref{mainthm3} is
complete.
\end{proof}

{\it Remark.} Let $\mu$ be a Borel measure on $X$ and let $\mathcal
J$ be the ideal in $\mathcal B(X)$ of measure zero sets. Then
$$
\Aut(X,\mathcal J)=G(X,\mathcal J)/I(X,\mathcal J)
$$
is the group of measure class preserving transformations on
$(X,\mu)$. If $\mu$ is a $\sigma$-finite measure then this is a
Polish group, usually denoted $\Aut^*(X,\mu)$, see
\cite[17.46]{kechris95}. The notions of near-actions and spatial
actions can be generalized in the obvious way to measure class
preserving actions, and Theorem \ref{mainthm3} applies to
$\Aut^*(X,\mu)$. Thus we have:

\begin{corollary}
Let $X$ be a standard Borel space and $\mu$ a Borel measure with at
least one uncountable null set (e.g. $X$ uncountable and $\mu$ a
$\sigma$-finite measure.) Let $G$ be a group of at most cardinality
$\aleph_1$. Then every measure class preserving near-action of $G$
on $(X,\mu)$ has a spatial model acting by Borel automorphisms.

In particular, if CH holds then all measure class preserving
near-actions have a spatial model acting by Borel automorphisms.
\end{corollary}

\section{A Lebesgue measurable lifting from MA}

In this section we consider the easier problem of producing a {\it
Lebesgue measurable} lifting of $\id:\Aut(X,\mu)\to\Aut(X,\mu)$,
where $(X,\mu)$ denotes a standard Borel probability space. Since it
presents no added difficulty in the proof, we will consider the more
general problem of producing a lifting of
$\id:\Aut^*(X,\mu)\to\Aut^*(X,\mu)$, where $\Aut^*(X,\mu)$ denotes
the group of measure class preserving automorphisms of $(X,\mu)$.

Let $G_{\lm}(X,\mu)$ denote the set of all Lebesgue measurable
measure class preserving bijections of $(X,\mu)$ which have a
Lebesgue measurable inverse, and let
$$
I_{\lm}(X,\mu)=\{f\in G_{\lm}(X,\mu): f(x)=x \text{ a.e.}\}.
$$
Then $I_{\lm} (X,\mu)$ is a subgroup of $G_{\lm}(X,\mu)$ such that
$$
G_{\lm}(X,\mu)/I_{\lm}(X,\mu)\simeq \Aut^*(X,\mu),
$$
and we denote by $\varphi_{\lm}:G_{\lm}(X,\mu)\to\Aut^*(X,\mu)$ the
canonical epimorphism with kernel $I_{\lm}(X,\mu)$.

It is worthwhile making explicit the similar category version of
this. If $Y$ is a topological space, one may consider the group $G_{\bm}(Y)$
of all {\it Baire measurable} bijections (see \cite[8.I]{kechris95}) which have a Baire
measurable inverse, and which preserve the ideal of meagre sets.
Then, in analogy with the above, we may form
$$
I_{\bm}(Y)=\{f\in G_{\bm}(Y): \{x\in Y: f(x)=x\} \text{ is
comeagre}\}
$$
and the corresponding quotient group
$$
G_{\bm}(Y)/I_{\bm}(Y)= \Aut_{\bm}(Y).
$$
Let $\varphi_{\bm}:G_{\bm}(Y)\to\Aut_{\bm}(Y)$ be the canonical
epimorphism.

In this section we prove:

\begin{theorem}
Suppose Martin's Axiom, MA, holds. Then $\id: \Aut^*(X,\mu)\to
\Aut^*(X,\mu)$ splits over $\varphi_{\lm}$, i.e. there is
$h:\Aut(X,\mu)\to G_{\lm}(X,\mu)$ such that $\varphi_{\lm}\circ
h=\id$.\label{lebesgue}

Similarly, assuming MA, if $Y$ is a locally compact c.c.c. Hausdorff
space (e.g. $Y$ a locally compact Polish space) then $\id:\Aut_{\bm}(Y)\to\Aut_{\bm}(Y)$ splits over
$\varphi_{\bm}$.
\end{theorem}

Theorem \ref{lebesgue} follows from the following more general
statement:

\begin{theorem}
Suppose $\kappa\leq 2^{\aleph_0}$ is a cardinal and $\ma(\lambda)$
holds for all cardinals $\lambda<\kappa$. Then for any group $H\leq
\Aut^*(X,\mu)$ of cardinality $\kappa$ there is a homomorphism $h:
H\to G_{\lm}(X,\mu)$ such that $\varphi_{\lm}\circ h=\id_H$.

Similarly for the category case.\label{lebesgue2}
\end{theorem}

{\it Remark.} The only role played by Martin's Axiom in the proof of
the above is to ensure that the union of fewer than $\kappa$ null
sets (meagre sets) in the standard Borel probability space $(X,\mu)$
(locally compact c.c.c. Hausdoff space $Y$) is again a null set
(meagre set), see \cite[Theorem 2.21 and 2.22]{kunen80}.

\medskip

Before proving \ref{lebesgue2}, we need two straightforward
generalizations of Lemma \ref{extend} and \ref{universal}.

\begin{lemma}
Let $\kappa< 2^{\aleph_0}$ and suppose that $\ma(\kappa)$ holds. Let
$H<\Aut^*(X,\mu)$ be a group of cardinality $\kappa$, and let
$T_0\in\Aut^*(X,\mu)\setminus H$. Suppose $\rho: H\times X\to X$ is
an action of $H$ by Lebesgue measurable bijections such that
$\varphi_{\lm}(\rho(T,\cdot))=T$ for all $T\in H$. If $\theta_0\in
G_{\lm}(X,\mu)$ is such that $\varphi_{\lm}(\theta_0)=T_0$ then
there is $Z\subseteq X$ such that $\mu(Z)=1$, $Z$ is invariant under
all $\theta_0$ and $\rho(T,\cdot)$ for all $T\in H$, and there
is an action $\tilde \rho:\langle H\cup\{T_0\}\rangle\times Z\to Z$
such that $\tilde\rho(T_0,\cdot)=\theta_0\restrict Z$ and $\tilde
\rho\restrict H\times Z = \rho\restrict H\times Z$.

The same statement holds, mutatis mutandis, in the Baire category
case.\label{extend2}
\end{lemma}

\begin{proof}
By the same argument given for \ref{extend}, the Lemma holds for $H$
countable. If $H$ is not countable, then for each countable subgroup
$\Delta<H$ we can find $Z_\Delta\subseteq X$ with $\mu(Z_\Delta)=1$
and satisfying the Lemma for $\Delta$ and $\rho\restrict\Delta\times
X$. Let $\tilde \rho_{\Delta}:\Delta\times Z_\Delta\to Z_\Delta$
witness this. If we let
$$
Z=\bigcap \{Z_\Delta:\Delta<H\ \text{ is a countable
subgroup}\}
$$
then since $\ma(\kappa)$ holds we have $\mu(Z)=1$ and
$$
\tilde \rho=\bigcup\{\tilde\rho_{\Delta}\restrict \Delta\times
Z:\Delta<H \text{ is a countable subgroup}\}
$$
is as required.
\end{proof}

The notion of shift action $\beta: H\actson X^G$ generalizes easily
to uncountable groups $H\leq G$.

\begin{lemma}
Let $\aleph_0\leq \kappa<2^{\aleph_0}$ and assume $\ma(\kappa)$
holds. Let $H\leq G$ be groups, $|G|=\kappa$, and let $X=2^\N$.
Then the shift action $\beta: H\actson X^G$ has the following
universality property: If $Y$ is a set, $|Y|\leq 2^{\aleph_0}$, and
$\sigma: H\actson Y$ is an action of $H$ on $Y$, then there is an
injection $\psi:Y\to X^G$ such that for all $h\in H$,
$$
\psi(\sigma(h)(x))=\beta(h)(\psi(x)).
$$\label{universal2}
\end{lemma}

\begin{proof}
The proof is identical to that of Lemma \ref{universal} once we note
that $Y$ is countably separated and that by $\ma(\kappa)$,
$2^\kappa=2^{\aleph_0}$.
\end{proof}

\begin{proof}[Proof of Theorem \ref{lebesgue2}]
We will construct an action $\rho: H\times X\to X$ by Lebesgue
measurable bijections by induction on $\kappa=|H|$ such that
$$
\varphi_{\lm}(\rho(T,\cdot))=T
$$
for all $T\in H$. The proof is made easier compared to the Borel
case by the fact that during the inductive construction, we only
need to make sure that for each $T\in H$, $\rho(T,\cdot)$ is defined
and fixed on a set of full measure from some point on. Below, if a
Lebesgue measurable measure class preserving bijection $\theta: Y\to
Y$ is defined only on a subset $Y\subseteq X$ of full measure, we
will allow ourselves to apply $\varphi_{\lm}$ to $\theta$, by
identifying $\theta$ with the Lebesgue measurable bijection
$\bar\theta:X\to X$ that acts identically on $X\setminus Y$.

Let $(T_\alpha:\alpha<\kappa)$ be an enumeration of $H$. Let
$$
H_\alpha=\langle T_\beta:\beta<\alpha\rangle,
$$
where as usual $H_0=\{I\}$. We decompose the space $X$ into disjoint
pieces $(X_\alpha:\alpha<\kappa)$, such that $\mu(X_0)=1$, and
$|X_\alpha|=2^{\aleph_0}$ for all $\alpha<\kappa$. For
$\alpha<\kappa$, we will inductively define Lebesgue measurable sets
$Y_\alpha\subseteq X$ and actions $\rho_\alpha: H_\alpha\times
Y_\alpha\to Y_\alpha$ by Lebesgue measurable bijections such that
\begin{enumerate}
\item $Y_0=X_0$
\item $\mu(Y_\alpha)=1$
\item $\beta<\alpha\implies Y_\beta\subseteq Y_\alpha$
\item $\beta<\alpha\implies \rho_\alpha\restrict H_\beta\times Y_\beta=\rho_\beta$
\item $(\forall T\in H_\alpha)\varphi_{\lm}(\rho_\alpha (T,\cdot))=T$
\item $\alpha<\gamma\implies Y_\alpha\cap X_\gamma=\emptyset$
\end{enumerate}
If we succeed in doing so then $\rho:H\times X\to X$ defined by
$$
\rho(T,x)=\left\{ \begin{array}{ll}
\rho_\alpha(T,x) & \textrm{if } T\in H_\alpha, x\in Y_\alpha,\\
x & \textrm{otherwise}\\
\end{array} \right.
$$
is as required. So suppose $\rho_\gamma$ and $Y_\gamma$ have been
defined for all $\gamma<\alpha$.

\medskip

Case 1. $\alpha$ is a limit ordinal.

\medskip

Then we let $Y_\alpha=\bigcup_{\gamma<\alpha} Y_\gamma$ and
$$
\rho_\alpha=\bigcup_{\gamma<\alpha} \rho_\gamma.
$$

\medskip

Case 2. $\alpha=\gamma+1$.

\medskip

Let $\theta: Y_\gamma\to Y_\gamma$ be a bijection such that
$\varphi_{\lm}(\theta)=T_\gamma$. Then by Lemma \ref{extend2} there
is some $Z\subseteq Y_\gamma$ such that $\mu(Z)=1$ and $Z$ is
invariant under $\theta$ and $\rho_\gamma(T,\cdot)$ for all $T\in
H_\gamma$, and an action $\rho^0: H_\alpha\times Z\to Z$ such that
$\rho^0(T_\gamma,\cdot)=\theta\restrict Z$ and
$$
\rho^0\restrict H_\gamma\times Z=\rho_\gamma\restrict H_\gamma\times
Z.
$$
Using Lemma \ref{universal2} we may find an injection $\psi:
Y_\gamma\setminus Z\to (2^\N)^{H_\alpha}$ such that
$$
\psi(\rho_\beta(T,x))=\beta(T,\psi(x))
$$
for all $x\in Y_\gamma\setminus Z$, where $\beta$ denotes the
shift-action of $H_\alpha$ on $(2^\N)^{H_\alpha}$. Extend $\psi$ to
a bijection $\bar\psi: W\cup Y_\gamma\setminus Z\to
(2^\N)^{H_\alpha}$  for some $W\subseteq X_{\gamma+1}$ such that $\bar\psi\restrict Y_\gamma\setminus
Z=\psi$, and define $\rho^1: H_\alpha\times W\cup Y_\gamma\setminus
Z\to W\cup Y_\gamma\setminus Z$ by
$$
\rho^1(T,x)=\bar\psi^{-1}(\beta(T,\bar\psi(x))).
$$
If we then let
$$
\rho_{\alpha}=\rho^0\cup\rho^1
$$
and $Y_\alpha=Y_\gamma\cup W$ then (1)--(6) above are satisfied.

The proof of the category case is similar.
\end{proof}

\section{Open problems}

Throughout this section, $(X,\mu)$ will denote a standard Borel
probability space.

In light of the results of this paper, as well as the strong analogy
between the lifting problem for $\Aut(X,\mu)$ (Question
\ref{mainquestion}) and the lifting problem for the measure algebra,
there are several open problems that suggest themselves. Surely the
most important open problem is the following:

\begin{problem}
Is it consistent with ZFC that $\id:\Aut(X,\mu)\to\Aut(X,\mu)$ does
\emph{not} split, i.e. that there is \emph{no} homomorphic lifting
$h:\Aut(X,\mu)\to G_{\mep}(X,\mu)$?\label{mainproblem}
\end{problem}

The analogous question for the measure algebra was settled in the
affirmative by Shelah, \cite{shelah83} and \cite{shelahpif}, and so
one naturally expects that Problem \ref{mainproblem} also has a
positive answer. This view seems to be further supported by the
result of Glasner, Tsirelson and Weiss \cite{gltswe05} that no
non-trivial Borel measure preserving near-action of $\Aut(X,\mu)$ on a
standard Borel probability space can have a Borel spatial
model.

It is worthwhile pointing out that the following na\"ive approach to
Problem \ref{mainproblem} \emph{does not} work: Let $\lambda$ denote
the Lebesgue measure on $\R$. For each $A\subseteq [0,1)$ Borel,
associate to it the measure preserving bijection $T_A:[0,2)\to
[0,2)$ given by
$$
T_A(x)=\left\{ \begin{array}{ll}
x+1 & \textrm{if $x\in A$}\\
x-1 & \textrm{if $x-1\in A$}\\
x & \textrm{otherwise.}
\end{array} \right.
$$
Then for each $[A]\in\malg([0,1),\lambda)$ this defines a unique
element of $\bar\varphi(T_A)$ in $\Aut([0,2),\lambda)$, and if we
consider $\malg([0,1),\lambda)$ as a group under the symmetric
difference operation, then $[A]\mapsto \bar\varphi(T_A)$ is a group
homomorphism.

One might now have hoped that if there was a lifting of
$\Aut([0,2),\lambda)$ then the above construction would lead to a
lifting of $\malg([0,1),\lambda)$. Ilijas Farah has pointed out that
this is not so, because \emph{as a group, $\malg([0,1),\lambda)$
always has a lifting}. This may be seen, for instance, by an easy
induction.

\medskip

The failure of this approach does prompt the following question:

\begin{problem}
What is the relationship between the lifting problem for
$\Aut(X,\mu)$ and the lifting problem for $\malg(X,\mu)$? In
particular, does a lifting for one imply a lifting for the other?
\end{problem}

The proof of Glasner, Tsirelson and Weiss relies on that the spatial
model implementing the given near-action is (jointly) Borel, i.e.
that the spatial action $\rho:G\times X\to X$ is Borel as a map from
$G\times X$ to $X$. The following natural question was brought up by
Solecki:

\begin{problem}
Does it follow from CH that every near-action has a
\emph{separately} Borel spatial model, i.e. a pointwise implementing
action $\rho:G\times X\to X$ such that for all $x\in X$ the map
$g\mapsto \rho(g,x)$ is Borel, and for all $g\in G$ the map
$x\mapsto \rho(g,x)$ is Borel?
\end{problem}

The proof of Theorem \ref{mainthm1} does not seem to give any
information about this. Solecki's question also brings up the more
general question of what sort of regularity properties are necessary
for Glasner-Tsirelson-Weiss' result to go through. As a ``test
case'' for Problem \ref{mainproblem} it would be of interest to know
the answer to the following:

\begin{problem}
Does it follow from AD, the Axiom of Determinacy (see e.g.
\cite{kechris95} or \cite{moschovakis80}), that there is \emph{no}
homomorphic lifting $h:\Aut(X,\mu)\to G_{\mep}(X,\mu)$? More
generally, does it follow from AD that no Levy group can act
pointwise in a non-trivial measure preserving way?
\end{problem}

Another question that arise from the proof of Theorem \ref{mainthm3}
is the problem of Borel hierarchy complexity of the automorphisms
constructed for the lifting. It is clear that nothing in the
construction produces a bound on the rank in the Borel hierarchy of
the Borel bijections produced. So we ask:

\begin{problem}
Is it consistent with ZFC to have a lifting $h:\Aut(X,\mu)\to
G_{\mep}(X,\mu)$ such that for some $\gamma<\omega_1$,
$h(T)\in\mathbf\Pi^0_\gamma$ for all $T\in\Aut(X,\mu)$? That is, can
we produce a lifting where all the Borel automorphisms have a bounded
Borel rank? If yes, does the existence of such a lifting follow from
CH?\label{bounded}
\end{problem}

The problem is analogous to a problem of A.H. Stone, see Problem
DO.c in \cite{fremlinproblems}, which asks the same for the lifting
problem for $\malg(X,\mu)$. This is (to my knowledge) still open for
$\malg(X,\mu)$.

\medskip

As mentioned in the introduction, Carlson, Frankiewicz and Zbierski
have shown that if we add $\aleph_2$ Cohen reals (or random reals)
to a model of CH, then $\malg(X,\mu)$ still has a lifting into
$\mathcal B(X)$. In particular, it is consistent with $\neg$CH to
have a lifting of $\malg(X,\mu)$. We ask analogously:

\begin{problem}
Is it consistent with $\neg$CH to have a lifting of $\Aut(X,\mu)$?
In particular, if we add $\aleph_2$ Cohen reals (or random reals) to
a model of CH, do we still have a lifting of $\Aut(X,\mu)$?
\end{problem}

\medskip

Our next problem concerns Theorem \ref{lebesgue}. In
\cite{maharam58}, Maharam shows that there always is a lifting of
$\malg(X,\mu)$ with \emph{Lebesgue measurable} sets. (Maharam
attributes this result to von Neumann.) In light of Theorem
\ref{lebesgue}, we therefore ask:

\begin{problem}
Does it follow from ZFC alone that there is a Lebesgue measurable
lifting of $\Aut(X,\mu)$?
\end{problem}

Finally, David Fremlin has asked me the following questions, which I
regrettably do not know the answer to.

\begin{problem}
In Theorem \ref{mainthm3}, can we dispense with the hypothesis of
$\mathcal J$ containing at least one uncountable Borel set? In
particular, does Theorem \ref{mainthm3} hold if we take $\mathcal J$
to be the set of all countable subsets of $X$?
\end{problem}

Fremlin's next question concerns semigroups rather than groups. If
$X$ is a Polish space,  $\mathcal J$  a  $\sigma$-ideal of $\mathcal
B(X)$, $\mathfrak A=\mathcal B(X)/\mathcal J$  and  $G$  is a
countable semigroup of Boolean homomorphisms from  $\mathfrak A$ to
itself, then in \cite[344B]{fremlin3}, it is shown that  $G$ can be
represented by a family  $f_{\pi}$ of Borel functions from $X$ to
itself such that $f_{\pi\phi}=f_{\phi}f_{\pi}$ for all $\pi,\phi\in
G$. Fremlin asks:

\begin{problem}
Does this Theorem hold for semigroups $G$ of cardinality $\aleph_1$?
\end{problem}

\bibliographystyle{amsplain}
\bibliography{mplift}

\end{document}